\documentclass[12pt,a4paper,reqno]{amsproc}
\usepackage{amsmath}
\usepackage{amssymb}
\usepackage{amsthm}
\usepackage[backrefs]{amsrefs}
\usepackage[german,danish,english]{babel}
\usepackage[utf8]{inputenc} 
\usepackage[T1]{fontenc}
\usepackage{graphicx}
\usepackage{color}
\usepackage{mathrsfs}
\usepackage{comment}
\usepackage{multirow}
\usepackage{enumitem}
\usepackage[footnote,draft,danish,silent]{fixme}
\usepackage{MnSymbol}

\newcommand{\NN}{\mathbb N}
\newcommand{\ZZ}{\mathbb Z}

\newcommand{\QQ}{\mathbb Q}

\newcommand{\tnorm}[1]{{\left\vert\kern-0.25ex\left\vert\kern-0.25ex\left\vert #1 
   \right\vert\kern-0.25ex\right\vert\kern-0.25ex\right\vert}}

\newtheorem{stn}{Sætning}[section]
\newtheorem{theorem}[stn]{Theorem}

\newtheorem{lemma}[stn]{Lemma}

\theoremstyle{definition}\newtheorem{defi}[stn]{Definition}
\theoremstyle{remark}\newtheorem{rem}[stn]{Remark}

\title{On the homotopy type of choice spaces}
\author{Shreyas Samaga}
\address{Shreyas Samaga }
\email{snshreyas@iiserb.ac.in}

\begin{document}
\begin{abstract}
We study continuous, symmetric and unanimous aggregation functions and continuous majority functions and prove that such functions exist on a choice space if and only if the choice space is contractible.
\end{abstract}

\maketitle

\begin{section}
{Introduction}
In this article, we will look at a few possible social choice functions like the symmetric and unanimous aggregation function and the majority function. Our aim is to find the required structure on the choice space for it to admit such functions. Such functions are used in Social Choice Theory. Social Choice Theory is a framework to aggregate or combine individuals' preferences to reach a collective decision. A symmetric aggregation function does not distinguish between individuals and gives equal weight to each person's preference as it is symmetric, i.e. independent of the ordering of the arguments of the function. A majority function is a type of aggregation function with the property that if more than half the arguments of the function are equal then the output of the function also equals that value. This is also an interesting function to study as it captures the preference of a majority among the individuals. Chichilnisky-Heal in \cite{Chichilnisky-Heal} observed that if the choice space is a finite CW-complex, a symmetric and unanimous aggregation function exists for any number of agents if and only if the choice space is contractible. Weinberger in \cite{Weinberger} strengthened the result and proved that if the choice space is a finite CW-complex and if a symmetric and unanimous aggregation function exists for some fixed number of individuals, then the choice space is contractible. In this article, we give a detailed proof of Weinberger's result in \cite{Weinberger} and also provide a simplified proof of the result proven in \cite{Chichilnisky-Heal}. Furthermore, we provide a proof for the existence of a majority function on a contractible choice space and prove that if a majority function exists for some $n$ on a locally finite CW-complex $X$, then $X$ is contractible which is a generalisation of this result which was known only for finite CW-complexes, as can be seen in section 3.9 of \cite{Taylor}.
\end{section}

\begin{section}{Preliminaries}

Let us recall some basic results from Algebraic Topology. These results can be found in \cite{Hatcher}.

\begin{theorem}[Whitehead]\label{Whitehead Theorem}
Let $X$ and $Y$ be CW-complexes. Let $f:X \to Y$ be a map such that it induces isomorphisms on all the homotopy groups, then $f$ is a homotopy equivalence.
\end{theorem}

\begin{rem}\label{Remark 1}
Let $X$ be a CW-complex. If $\pi_i(X)\cong 0$  $\forall$ $0\leqslant i\leqslant \infty$ then $X$ is contractible.\\
Indeed, consider a map $f:X \to \{p\}$ where $\{p\}$ is a point. It is clear that $\pi_i(\{p\}) = 0$ $\forall i$. We also have induced isomorphisms between the homotopy groups of $X$ and $\{p\}$. From Whitehead's theorem, we get that $f$ is a homotopy equivalence. Thus, $X$ is contractible.
\end{rem}

\begin{defi}[Cellular Map]
Let $X$ and $Y$ be CW-complexes. Any map $f:X \to Y$ such that $f(X^{(n)}) \subset Y^{(n)}$ $\forall n$ is a cellular map.
\end{defi}

\begin{theorem} \label{Cellular Approximation}
Let $X$ and $Y$ be CW-complexes. Then every map $f:X \to Y$ is homotopic to a cellular map.
\end{theorem}

\begin{rem}\label{Remark 2}
An immediate consequence of the above theorem is that the fundamental group of a finite CW-complex $X$ is finitely generated because the map $\gamma \colon S^1 \to X$ is a composition of the maps $\alpha: S^1 \to X^{(1)}$ and the inclusion map $i:X^{(1)} \to X$. $S^1$ here is given the CW structure with one $0$-cell and one $1$-cell. As we know that the elements of the fundamental group are the homotopy classes of loops, it is clear from the above argument that if there are finitely many $1$-cells in $X$ then the fundamental group also has to be finitely generated.
\end{rem}

\begin{theorem}[Hurewicz]\label{Hurewicz Theorem}
Let $X$ be an ($n-1$) connected space, $n>1$. Then $\tilde{H_i}=0$ $\forall i<n$ and $\pi_n(X) \cong H_n(X)$.
\end{theorem}

\begin{lemma}\label{Cellular Homology}
If $X$ is a CW-complex, then:
\begin{enumerate}
\item$H_k(X^{(n)}, X^{(n-1)}) = 0$ $\forall k \neq n$ and is free abelian for $k=n$, with a basis in correspondence with the $n$-cells of $X$.
\item$H_k(X^{(n)})=0$ $\forall k>n$.
\item The map from $H_k(X^{(n)})$ to $H_k(X)$ induced by the inclusion of $X^{(n)}$ in $X$ is an isomorphism for $k<n$.
\end{enumerate}
\end{lemma}

Now, we recall that for a CW-complex, the cellular homology groups are isomorphic to the ordinary singular homology groups. The reader can refer to \cite{Hatcher} for definitions of singular homology groups and cellular homology groups.

\begin{theorem}
Let $X$ be a CW-complex. $H_n^{CW}(X) \cong H_n(X)$.
\end{theorem}

\begin{rem}\label {Remark 3}
It is clear from the above theorem that 
\begin{enumerate}
\item$H_n(X) = 0$ if $X$ has no $n$-cells.
\item If $X$ is a CW-complex with $k$ $n$-cells, then $H_n(X)$ is generated by at most $k$ generators.
\end{enumerate}
\end{rem}

The following result can be found in \cite{Sibe}

\begin{theorem}\label{Metrizability}
Every locally finite CW-complex is metrizable.
\end{theorem}

The following results can be found in Chapter 6 of \cite{Sakai}.

\begin{defi}
Let $X$ be a topological space. A \textbf{neighbourhood retract} of $X$ is a closed set in $X$ which is a retract of some neighbourhood in $X$. A metrizable space $X$ is called an \textbf{absolute neighbourhood retract (ANR)} if $X$ is a neighbourhood retract of every arbitrary metrizable space that contains $X$ as a closed subspace. A metrizable space $X$ is called an \textbf{absolute retract (AR)} if $X$ is a retract of every arbitrary metrizable space that contains $X$ as a closed subspace.   
\end{defi}

\begin{defi}
A space $Y$ is called an \textbf{absolute neighbourhood extensor for metrizable spaces (ANE)} if each map $f:A\to Y$ from any closed set $A$ in an arbitrary metrizable space $X$ extends over some neighbourhood $U$ of $A$ in $X$. When $f$ can be extended over $X$, we call $Y$ an \textbf{absolute extensor for metrizable spaces (AE)}.
\end{defi}

\begin{theorem}\label{ANR}
Every locally finite CW-complex is an ANR.
\end{theorem}

\begin{theorem}\label{ANE}
Let $X$ be a metrizable space. Then
\begin{enumerate}
\item$X$ is ANR if and only if $X$ is ANE.
\item$X$ is AR if and only if $X$ is AE.
\end{enumerate}
\end{theorem}

\begin{theorem}\label{AR}
A metrizable space is an AR if and only if it is a contractible ANR.
\end{theorem}

\begin{rem}\label{Remark 4}
From the above theorems, it becomes clear that every locally finite CW-complex is an ANR and is metrizable and hence is an ANE and that a contractible locally finite CW-complex is an AR and hence an AE.
\end{rem}

\end{section}

\begin{section}{Main Results}

\begin{defi}
Let $X$ be a topological space which is considered as the choice space for each individual. Let there be $n$ individuals. Then, a \textbf{symmetric and unanimous aggregation function} \\
$A:X^n \to X$ is a continuous function such that
\begin{enumerate}
\item$A(x_1,x_2,...,x_n) = A(x_{\sigma(1)},x_{\sigma(2)},...,x_{\sigma(n)})$ $\forall \sigma \in S_n$, the permutation group. This property is called anonymity.
\item$A(x,x,...,x) = x$. This is called unanimity.
\end{enumerate}
\end{defi}

\begin{lemma}[Weinberger]\label{Weinberger Lemma}
If $X$ is a connected CW-complex such that for some $n > 1$ there is a map A as stated above, then the fundamental group is abelian and all the homotopy groups are uniquely $n$-divisible. 
\end{lemma}

We, now, present a simplified version of his proof.
\begin{proof}
Consider the induced homomorphism, $A_* : \pi_1(X^n) \to \pi_1(X)$. Now, $\pi_1(X^n) \cong (\pi_1(X))^n$. Due to the property of anonymity, $A_*$ is the same on each of its factors, call it $\rho_* : \pi_1(X) \to \pi_1(X)$. This is because $A_*$ is also a symmetric function. To see this, let $\gamma$ be a loop in $\pi_1(X^n)$. Then $\gamma = (\gamma_1,\gamma_2,...,\gamma_n)$ where each $\gamma_i : S^1 \to X$ is a loop in $\pi_1(X)$. $A_*([\gamma_1],[\gamma_2],...,[\gamma_n]) = A_*([\gamma]) = [A\gamma] = [A(\gamma_1,\gamma_2,...,\gamma_n)] = [A(\gamma_2,\gamma_1,...,\gamma_n)] = A_*([\gamma_2],[\gamma_1],...,[\gamma_n])$. This is true for swapping any two coordinates and for any reordering of arguments of the function as $A$ is a symmetric function, proving that $A_\star$ is also a symmetric function and hence is the same on each of its factors.
\\
Now, $A\Delta = I$, where $\Delta$ denotes the diagonal map $\Delta : X \to X^n$. Since $\pi_1$ is a functor, $A_* \Delta_* = I$ implying that $A_{*}$ is surjective. 
Also, $\Delta_* = {i_1}_{*} + {i_2}_{*} + ... {i_n}_{*}$, where ${i_j}_{*}$ denotes the induced homomorphism due to the inclusion map in the jth coordinate. Therefore, $A_* \Delta_* = A_* {i_1}_{*} + A_* {i_2}_{*} + ... + A_* {i_n}_{*}$. Since $A_*$ is the same on each of its coordinates, we get the equation $n\rho_* = I$. This provides the unique $n$-divisibility of the fundamental group and also that $\rho_*$ is surjective.
\\
The individual summands of $\pi_1(X^n)$ commute. Hence their images under $A_*$ also commute as $A_*$ is a group homomorphism. Further, $A_*$ and $\rho_*$ being surjective homomorphisms give us that $\pi_1(X)$ is abelian.  
The arguments used can be applied to the higher homotopy groups as well, inferring that all the higher homotopy groups are also uniquely $n$-divisible.
\end{proof}

\begin{theorem}
Let $X$ be a finite connected CW-complex. Assume that for all $n$, $A$, as defined above, exists. Then $X$ is contractible. 
\end{theorem}

\begin{proof}
We know that $\pi_i(X)$ are uniquely $n$-divisible for all $n$ and $i$. Hence, $\pi_i(X)$ $\forall i>1$ form a vector space over the field of rational numbers. Hence, $\pi_i(X) \cong \bigoplus_{J_i} \mathbb{Q}$ $\forall i$ for some indexing sets $J_i$. We know that $\mathbb{Q}$ is not finitely generated as an abelian group and thus, $\bigoplus_{J_i}\mathbb{Q}$ is not finitely generated as an abelian group when $J_i \neq \phi$.

Assume $\pi_1(X)$ is not trivial. We have seen in Remark \ref{Remark 2} that the fundamental group of a finite CW complex is finitely generated as an abelian group. From Lemma \ref{Weinberger Lemma}, we get that $\pi_1(X) \cong \bigoplus_{J_1} \QQ$ for some indexing set $J_1$ which says that $\pi_1(X)$ is not finitely generated as an abelian group, which is a contradiction. Therefore, $\pi_1(X) \cong 0$.

Assume $\pi_2(X)$ is not trivial. Since $\pi_1(X)\cong 0$, by Theorem \ref{Hurewicz Theorem}, we get $\pi_2(X) \cong H_2(X)$. If $\pi_2(X) \cong \bigoplus_{J_2} \QQ$ for some indexing set $J_2(\neq \phi)$, it is not finitely generated as a group over $\ZZ$ implying that $H_2(X)$ is not finitely generated over $\ZZ$, which is a contradiction as we know from Theorem \ref{Cellular Homology} and Remark \ref{Remark 3} that the homology groups of a finite CW complex are finitely generated over $\ZZ$. Thus, $\pi_2(X) \cong 0$. Similarly, by induction over $n$, we get that $\pi_i(X) \cong 0$ $\forall i$. From Remark \ref{Remark 1} we get that $X$ is contractible.
\end{proof}

In the case that $X$ is not a finite CW-complex, Weinberger proved that it is of the homotopy type of a product of rational Eilenberg-MacLane spaces. Eilenberg-MacLane spaces are the spaces with at most a single non-vanishing homotopy group. If the homotopy group is abelian and uniquely divisible by $n$ for all $n$, then it is a vector space over the rational numbers. In such cases, an Eilenberg-MacLane space is called a rational Eilenberg-MacLane space. 

The converse of this theorem can be proven by a similar approach as that developed in the next proof that we give.

We now look at another choice function, namely, a majority function. A majority function is defined as follows:
\begin{defi}
Let $X$ be the choice space of each agent and let there be $n$ agents. Then, a \textbf{majority function} is a continuous function $M:X^n \to X$ such that \\
$M(x_1,x_2,...,x_n)=y$ if more than half of the arguments of the function take the value $y$.
\end{defi}

\begin{theorem}
Let $X$ be a locally finite connected CW-complex. If X is contractible, then there exists a majority function for any $n \in \NN$.
\end{theorem}

\begin{proof}
Let $A$ be the subspace of $X^n$ which consists of points which have the same coordinate in at least $[n/2]+1$ coordinates where $[ ]$ denotes the greatest integer function. Let $M':A \to X$ be defined on $A$ s.t $M'$ is a majority function on $A$. Now, we need to extend this to $X^n$ such that the resulting function is also continuous.Since $X$ is contractible, and thus $X^n$ is contractible, from Remark \ref{Remark 4}, it is clear that $X^n$ is an AE. Thus, it suffices to check that $A$ is closed in $X^n$.

Consider $x\in X^n \setminus A$. Let $x=(x_1,x_2,...,x_n)$.
Let $x$ have $k$ distinct coordinate values. Let $x'_1,x'_2,...,x'_k \in X$ be the distinct coordinate values. It is clear that $x$ can have at most $[n/2]$ equal coordinate values. Consider a surjective function $\phi \colon \{1,2,...,n\} \to \{1,2,...,k\}$ such that $\phi(i)=j$ if $x_i=x'_j$ $\forall 1\leqslant i,j \leqslant n$. Since $X$ is Hausdorff, there are disjoint open sets $U_j$ separating the distinct coordinate values from each other. Let $U = \prod_{i = 1}^{n} U_{\phi(i)}$. We claim that $U \subset X^n \setminus A.$ 
\\
Consider $y \in U$. $y$ cannot have more than $[n/2]$ coordinates equal because if $y$ did have even $[n/2]+1$ coordinates equal, it would imply that $\exists y_s$, $0 \leqslant s \leqslant n$ such that $y_s \in U_i \cap U_j$,$i \neq j$, which is not possible as the open sets are disjoint. Therefore, $U\subset X^n \setminus A$ and clearly $U$ is open. Thus, $X^n \setminus A$ is open in $X^n$ and hence, $A$ is closed in $X^n$.
\end{proof}

We now prove the converse of this theorem in the general setting of a locally finite CW-complex. 

\begin{theorem}
Let $X$ be a locally finite connected CW-complex. Let $M$ be a majority function on X for some $n \in \NN$. Then $X$ is contractible.
\end{theorem}

\begin{proof}
Consider the induced homomorphism $M_\star$ on the fundamental groups. $M_\star: \pi_1(X^n) \to \pi_1(X)$. Now, $\pi_1(X^n) \cong (\pi_1(X))^n$. By the definition of the induced homomorphism,
$M_\star[\gamma] = [M\circ\gamma]$ where $\gamma : S^1 \to X^n$. Let $f:(\pi_1(X))^n \to \pi_1(X^n)$ be the isomorphism. $[\gamma] = f([\gamma_1], [\gamma_2],...,[\gamma_n])$ where each $\gamma_i : S^1 \to X$. It is clear that $M_\star([\gamma_1], [\gamma_2],...,[\gamma_n]) = [\tau]$ if $[\gamma_i] = [\tau]$ for more than $[n/2]$ arguments of the function. and hence $M_\star f$ is also a majority function. For, $M_\star f[\gamma] = [M \circ f(\gamma_1,\gamma_2,...,\gamma_n)]$ and $M$ is a majority function.\\
Assume that $\pi_1(X)$ is not trivial. Let $a\in \pi_1(X)$ be non-trivial. Then, $M_\star f(a,0,...,0)=0$ as $M_\star$ is a majority function. This is clearly true if $a$ is in any coordinate and the rest are zero. Now, we add all such distinct tuples to get $M_\star f(a,a,a,...,a) = 0$, as $M_\star f$ is a homomorphism. But $M_\star f(a,a,...,a)=a$ as $M_\star f$ is a majority function. This is a contradiction. Hence, $\pi_1(X)$ is trivial. The above argument applies to all the higher homotopy groups proving that $\pi_i(X)$ is trivial $\forall i$. Since $X$ is a connected CW-complex and all the homotopy groups are trivial, $X$ is contractible, as seen in Remark \ref{Remark 1}.

\end{proof}
\end{section}

\section*{Acknowledgements}
I would like to thank Prof. Dr. Andreas Thom for offering me this project and being my mentor for the same. I would also like to thank Dr. Martin Nitsche for guiding me through the project.

\end{document}